\documentclass[12pt,oneside,reqno]{amsart}

\usepackage{amsmath,amssymb,amsthm,textcomp}
\usepackage{amsfonts,graphicx}
\usepackage[mathscr]{eucal}
\usepackage{color}
\usepackage{csquotes}
\usepackage[backend=bibtex,%
firstinits=true,%
doi=false,%
isbn=true,%
url=false,%
maxnames=99]{biblatex}%

\AtEveryBibitem{\clearfield{issn}}
\AtEveryCitekey{\clearfield{issn}}
\numberwithin{equation}{section}
\DeclareNameAlias{sortname}{last-first}
\theoremstyle{definition}
\usepackage{mathtools}
\numberwithin{equation}{section}

\newcommand{\ncom}{\newcommand}

\ncom{\beq}{\begin{equation}}
\ncom{\eeq}{\end{equation}}
\ncom{\bea}{\begin{eqnarray*}}
\ncom{\eea}{\end{eqnarray*}}
\ncom{\beqa}{\begin{eqnarray}}
\ncom{\eeqa}{\end{eqnarray}}
\ncom{\nno}{\nonumber}
\ncom{\non}{\nonumber}
\ncom{\ds}{\displaystyle}
\ncom{\half}{\frac{1}{2}}
\ncom{\mbx}{\makebox{.25cm}}
\ncom{\hs}{\mbox{\hspace{.25cm}}}
\ncom{\rar}{\rightarrow}
\ncom{\Rar}{\Rightarrow}
\ncom{\noin}{\noindent}
\ncom{\bc}{\begin{center}}
\ncom{\ec}{\end{center}}
\ncom{\sz}{\scriptsize}
\ncom{\rf}{\ref}
\ncom{\s}{\sqrt{2}}
\ncom{\sgm}{\sigma}
\ncom{\Sgm}{\Sigma}
\ncom{\psgm}{\sigma^{\prime}}
\ncom{\dt}{\delta}
\ncom{\Dt}{\Delta}
\ncom{\lmd}{\lambda}
\ncom{\Lmd}{\Lambda}
\ncom{\Th}{\Theta}
\ncom{\e}{\eta}
\ncom{\eps}{\epsilon}
\ncom{\pcc}{\stackrel{P}{>}}
\ncom{\lp}{\stackrel{L_{p}}{>}}
\ncom{\dist}{{\rm\,dist}}
\ncom{\sspan}{{\rm\,span}}
\ncom{\re}{{\rm Re\,}}
\ncom{\im}{{\rm Im\,}}
\ncom{\sgn}{{\rm sgn\,}}
\ncom{\ba}{\begin{array}}
\ncom{\ea}{\end{array}}
\ncom{\hone}{\mbox{\hspace{1em}}}
\ncom{\htwo}{\mbox{\hspace{2em}}}
\ncom{\hthree}{\mbox{\hspace{3em}}}
\ncom{\hfour}{\mbox{\hspace{4em}}}
\ncom{\vone}{\vskip 2ex}
\ncom{\vtwo}{\vskip 4ex}
\ncom{\vonee}{\vskip 1.5ex}
\ncom{\vthree}{\vskip 6ex}
\ncom{\vfour}{\vspace*{8ex}}
\ncom{\norm}{\|\;\;\|}
\ncom{\integ}[4]{\int_{#1}^{#2}\,{#3}\,d{#4}}
\ncom{\vspan}[1]{{{\rm\,span}\{ #1 \}}}
\ncom{\dm}[1]{ {\displaystyle{#1} } }
\ncom{\ri}[1]{{#1} \index{#1}}

\newtheorem{theorem}{\bf Theorem}[section]
\newtheorem{remark}{\bf Remark}[section]
\newtheorem{proposition}{Proposition}[section]
\newtheorem{lemma}{Lemma}[section]

\newtheoremstyle
    {remarkstyle}
    {}
    {11pt}
    {}
    {}
    {\bfseries}
    {:}
    {     }
    {\thmname{#1} \thmnumber{#2} }

\theoremstyle{remarkstyle}

\def\eps{\varepsilon}

\begin{document}
\title{Generalized fractional birth process}
\author[Kuldeep Kumar Kataria]{Kuldeep Kumar Kataria}
\address{Kuldeep Kumar Kataria, Department of Mathematics, Indian Institute of Technology Bhilai, Raipur 492015, India.}
 \email{kuldeepk@iitbhilai.ac.in}
\author[Mostafizar Khandakar]{Mostafizar Khandakar}
\address{Mostafizar Khandakar, Department of Mathematics, Indian Institute of Technology Bhilai, Raipur 492015, India.}
\email{mostafizark@iitbhilai.ac.in}
\date{October 04, 2021}
\begin{abstract}
In this paper, we introduce a generalized birth process (GBP) which performs jumps of size $1,2,\dots,k$ whose rates depend on the state of the process at time $t\geq0$. We derive a non-exploding condition for it. The system of differential equations that governs its state probabilities is obtained. In this governing system of differential equations, we replace the first order derivative with Caputo fractional derivative to obtain a fractional variant of the GBP, namely, the generalized fractional birth process (GFBP). The Laplace transform of the state probabilities of this fractional variant is obtained whose inversion yields its one-dimensional distribution. It is shown that the GFBP is equal in distribution to a time-changed version of the GBP, and this result is used to obtain a non-exploding condition for it. A limiting case of the GFBP is considered in which jump of any size $j\ge1$ is possible. Also, we discuss a state dependent version of it.
\end{abstract}

\maketitle
\section{Introduction}
The pure birth process is a counting process in which the jump of unit size occurs in an infinitesimal interval of length $h$ with probability $\lambda_n h+o(h)$, that is, the jump intensity $\lambda_n>0$ depends on the state of the process at epoch $t\geq0$. Orsingher and Polito (2010) introduced and studied a fractional version of the pure birth process, namely, the fractional pure birth process (FPBP). It is denoted by $\{\mathcal{N}^{\alpha}(t)\}_{t\ge0}$, $0<\alpha\le 1$ whose state probabilities $q^{\alpha}(n,t)=\mathrm{Pr}\{\mathcal{N}^{\alpha}(t)=n\}$, $n\ge1$ solve the following system of fractional differential equations:
\begin{equation*}
\frac{\mathrm{d}^{\alpha}}{\mathrm{d}t^{\alpha}}q^{\alpha}(n,t)=-\lambda_{n}q^{\alpha}(n,t)+\lambda_{n-1}q^{\alpha}(n-1,t),
\end{equation*}
with the initial conditions
\begin{equation*}
q^{\alpha}(n,0)=\begin{cases}
1,\ \ n=1,\\
0,\ \ n\ge2.
\end{cases}
\end{equation*}
Here,  $\frac{\mathrm{d}^{\alpha}}{\mathrm{d}t^{\alpha}}$ is the Caputo fractional derivative defined  as (see Kilbas {\it et al.} (2006))
\begin{equation*}
\frac{\mathrm{d}^{\alpha}}{\mathrm{d}t^{\alpha}}f(t):=\left\{
\begin{array}{ll}
\dfrac{1}{\Gamma{(1-\alpha)}}\displaystyle\int^t_{0} (t-s)^{-\alpha}f'(s)\,\mathrm{d}s,\ \ 0<\alpha<1,\\\\
f'(t),\ \ \alpha=1.
\end{array}
\right.
\end{equation*}
Its Laplace transform is given by (see Kilbas et al. (2006), Eq. (5.3.3))
\begin{equation}\label{laplace}
\mathcal{L}\left(\frac{\mathrm{d}^{\alpha}}{\mathrm{d}t^{\alpha}}f(t);s\right)=s^{\alpha}\tilde{f}(s)-s^{\alpha-1}f(0),\ \ s>0.
\end{equation}
For $\alpha=1$, the FPBP reduces to the pure birth process (see Feller (1968)). 

The Poisson process is another important pure birth type process mainly used in modeling the count data. Its jump intensity is a constant $\lambda>0$, that is, it does not depend on the state of the process. It's a L\'evy as well as a renewal process with exponentially distributed waiting times. It has certain limitations in modeling the phenomena with long memory due to its light-tailed distributed waiting times. To overcome these limitations, many researchers have introduced and studied various fractional generalizatons of the Poisson process in the past two decades. These fractional variants are generally characterized by their heavy-tailed distributed waiting times. Some commonly used fractional versions of the Poisson process are the time fractional Poisson process (TFPP) (see Beghin and Orsingher (2009)), the space fractional Poisson process (SFPP) (see Orsingher and Polito (2012)) and their generalization, {\it viz.}, the space-time fractional Poisson process (STFPP) (see Orsingher and Polito (2012)) {\it etc}.

Di Crescenzo {\it et al.} (2016) introduced and studied a generalization of the Poisson process, namely, the generalized counting process (GCP) which performs $k$ kinds of jumps of amplitude $1,2,\dots,k$ with positive rates $\lambda_{1}, \lambda_{2},\dots,\lambda_{k}$, respectively. It is important to note that the jump intensities in GCP do not depend on its state. For $k=1$, the GCP reduces to the Poisson process. Moreover, they considered a fractional variant of the GCP, namely, the generalized fractional counting process (GFCP) by taking the Caputo fractional derivative in the system of differential equations that governs the state probabilities of the GCP. A limiting case of the GFCP, namely, the convoluted fractional Poisson process (CFPP) is obtained by letting $k\to \infty$ and suitably choosing the intensity parameters (see Kataria and Khandakar (2021a)). For $k=1$, the GFCP reduces to the TFPP. For more properties and special cases of the GFCP, we refer the reader to Kataria and Khandakar (2021b).

In this paper, we introduce a generalized birth process (GBP) $\{\mathcal{X}(t)\}_{t\ge0}$ which performs jumps of size $1,2,\dots,k$ with positive rates $\lambda_{(n)_{1}}, \lambda_{(n)_{2}},\dots,\lambda_{(n)_{k}}$, respectively. Here, $n$ denotes the state of the process at any given time $t\ge0$. We obtain a condition of no explosion for the GBP. We consider a fractional version of it, namely,  the generalized fractional birth process (GFBP)  by taking Caputo fractional derivative in the system of differential equations thats governs the state probabilities of GBP. We denote it by $\{\mathcal{X}^{\alpha}(t)\}_{t\ge0}$, $0<\alpha \le 1$ and its
state probabilities  $p^{\alpha}(n,t)=\mathrm{Pr}\{\mathcal{X}^{\alpha}(t)=n\}$ satisfy the following system of fractional differential equations:
\begin{equation*}
\frac{\mathrm{d}^{\alpha}}{\mathrm{d}t^{\alpha}}p^{\alpha}(n,t)=-\sum_{i=1}^{k}\lambda_{(n)_{i}} p^{\alpha}(n,t)+	\sum_{i=1}^{\min\{n,k\}}\lambda_{(n-i)_{i}}p^{\alpha}(n-i,t),\ \ n\ge n_{0},
\end{equation*}
with the initial conditions
\begin{equation*}
p^{\alpha}(n,0)=\begin{cases}
1,\ \ n=n_{0},\\
0,\ \ n>n_{0}.
\end{cases}
\end{equation*}
Here, $n_0$ is a fixed non-negative integer. The GFBP reduces to FPBP for $n_{0}=k=1$ and distinct $\lambda_{(n)_{1}}$'s for all $n\ge1$. It is shown that the TFPP, SFPP, STFPP, GFCP and CFPP are its other particular cases. We obtain the Laplace transform of its state probabilities whose inversion yields its one-dimensional distribution.  We establish the following time-changed relationship between the GBP and GFBP:
\begin{equation*}
\mathcal{X}^{\alpha}(t)\overset{d}{=}\mathcal{X}(T_{2\alpha}(t)),
\end{equation*}
where $\{T_{2\alpha}(t)\}_{t>0}$ is a random process independent of the GBP whose distribution solves (\ref{diff}). A limiting case of the GFBP is discussed which has jumps of any size $j\ge 1$. We obtain the state probabilities and a time-changed relation for this special case. Later, we discuss a state dependent version of the GFBP.
\section{Generalized fractional birth process}\label{section3}
In this section, we introduce a generalized birth process (GBP) which performs jumps of size $1,2,\dots,k$ with positive rates $\lambda_{(n)_{1}}, \lambda_{(n)_{2}},\dots,\lambda_{(n)_{k}}$, respectively. Here, $n$ denotes the state of the process at a given time $t\geq0$. We denote the GBP by $\{\mathcal{X}(t)\}_{t\ge0}$ and assume that it has $n_{0}$ progenitors at time $t=0$, that is, $\mathrm{Pr}\{\mathcal{X}(0)=n_{0}\}=1$, where $n_{0}\ge 0$ is a fixed integer. The probability of its jumps in an infinitesimal interval of length $h$ is given by 
\begin{equation*}
\mathrm{Pr}\{\mathcal{X}(t+h)=n+i|\mathcal{X}(t)=n\}=\begin{cases*}
1-\sum_{i=1}^{k}\lambda_{(n)_{i}}h+o(h),\ \ i=0,\\
\lambda_{(n)_{i}}h+o(h),  \  i=1,2,\dots,k,\\
o(h),\ \ i>k.
\end{cases*}
\end{equation*}
Note that the rates of jumps depend on the actual state of the process at time $t$ and thus the transition probabilitiy is a function of states and jump sizes.

 Let us denote $p(n,t)=\mathrm{Pr}\{\mathcal{X}(t)=n\}$, $n\ge n_{0}$. With these assumptions, we have 
 \begin{equation*}
 p(n,t+h)=p(n,t)\left(1-\sum_{i=1}^{k}\lambda_{(n)_{i}}h\right)+\sum_{i=1}^{\min\{n,k\}}\lambda_{(n-i)_{i}}p(n-i,t)h+o(h)
 \end{equation*}
 which reduces to the following form as $h\to 0$:
\begin{equation}\label{model1}
\frac{\mathrm{d}}{\mathrm{d}t}p(n,t)=-\sum_{i=1}^{k}\lambda_{(n)_{i}} p(n,t)+	\sum_{i=1}^{\min\{n,k\}}\lambda_{(n-i)_{i}}p(n-i,t),\ \ n\ge n_{0},
\end{equation}
with the initial conditions
\begin{equation*}
p(n,0)=\begin{cases}
1,\ \ n=n_{0},\\
0,\ \ n>n_{0}.
\end{cases}
\end{equation*}

A fractional version of the GBP, namely,  the generalized fractional birth process (GFBP) can be obtained by taking Caputo fractional derivative in the system of differential equations  given in (\ref{model1}). We denote it by $\{\mathcal{X}^{\alpha}(t)\}_{t\ge0}$, $0<\alpha \le 1$, and its
 state probabilities  $p^{\alpha}(n,t)=\mathrm{Pr}\{\mathcal{X}^{\alpha}(t)=n\}$ satisfy the following system of fractional differential equations:
\begin{equation}\label{model2}
\frac{\mathrm{d}^{\alpha}}{\mathrm{d}t^{\alpha}}p^{\alpha}(n,t)=-\sum_{i=1}^{k}\lambda_{(n)_{i}} p^{\alpha}(n,t)+	\sum_{i=1}^{\min\{n,k\}}\lambda_{(n-i)_{i}}p^{\alpha}(n-i,t),\ \ n\ge n_{0},
\end{equation}
with the initial conditions
\begin{equation*}
p^{\alpha}(n,0)=\begin{cases}
1,\ \ n=n_{0},\\
0,\ \ n>n_{0}.
\end{cases}
\end{equation*}
Equivalently,
\begin{equation}\label{model5}
\frac{\mathrm{d}^{\alpha}}{\mathrm{d}t^{\alpha}}p^{\alpha}(n,t)=-\sum_{i=1}^{k}\lambda_{(n)_{i}} p^{\alpha}(n,t)+	\sum_{i=1}^{k}\lambda_{(n-i)_{i}}p^{\alpha}(n-i,t),\ \ n\ge n_{0},
\end{equation}
 as $p^{\alpha}(n-i,t)=0$ for all $i>n$. For $\alpha=1$, we get an equivalent expression for the GBP.
 
\begin{remark} 
	
 The TFPP is obtained as a particular case of the GFBP
 for $n_{0}=0$, $k=1$ and $\lambda_{(n)_{1}}=\lambda$ for all $n\ge0$. The GFBP reduces to  the FPBP when $n_{0}=k=1$ and $\lambda_{(n)_{1}}$'s are distinct for all $n\ge1$. Also, for $n_{0}=0$, $\lambda_{(n)_{i}}=\lambda_{i}$ for all $n\ge0$ and $1\le i \le k$, the GFBP reduces to the GFCP. In case of infinite jumps, that is, letting $k\to \infty$ and  taking $\lambda_{(n)_{i}}=\lambda_{i}=\beta_{i-1}-\beta_{i}$, $i\ge 1$, $n\ge0$, the GFCP further reduces to the CFPP, where 
  $\{\beta_{i}\}_{i\in\mathbb{Z}}$ be a sequence of intensity parameters such that $\beta_{i}=0$ for all $i<0$ and $\beta_{i}>\beta_{i+1}>0$ for all $i\geq0$ with $\lim\limits_{i\to\infty}\beta_{i+1}/\beta_{i}<1$. However, if we choose $\lambda_{(n)_{i}}=\lambda_{i}=(-1)^{i+1}\lambda^{\beta}\beta(\beta-1)\cdots(\beta-i+1)/i!$, $i\ge 1$, $n\ge0$ for some $0<\beta\leq1$ and letting $k\to \infty$, we get the STFPP which reduces to the SFPP for $\alpha=1$.
 \end{remark}
 
The next result gives the sufficient condition for no explosion in GBP.

 \begin{theorem}\label{explosion}
 	If $\sum_{m=n_{0}}^{\infty}\left(\sum_{i=1}^{k}\sum_{j=1}^{i}\lambda^{2}_{(m-j+1)_{i}}\right)^{-1/2}=\infty$ then $\sum_{m=n_{0}}^{\infty}p(m,t)=1$.
 \end{theorem}
 \begin{proof}
 	Let $S(n_{0}+n,t)=\sum_{m=n_{0}}^{n_{0}+n}p(m,t)$, $n\ge 0$. On taking $\alpha=1$ in the system of differential equations  (\ref{model5}) and adding its first $n+1$ equations, we get
 	\begin{equation*}
 	\sum_{m=n_{0}}^{n_{0}+n}\frac{\mathrm{d}}{\mathrm{d}t}p(m,t)=-\sum_{m=n_{0}}^{n_{0}+n}\sum_{i=1}^{k}\lambda_{(m)_{i}} p(m,t)+	\sum_{m=n_{0}}^{n_{0}+n}\sum_{i=1}^{k}\lambda_{(m-i)_{i}}p(m-i,t)
 	\end{equation*}
 	which reduces to
 	\begin{equation*}
 	\frac{\mathrm{d}}{\mathrm{d}t}S(n_{0}+n,t)=-\sum_{j=1}^{k}\sum_{i=j}^{k}\lambda_{(n_{0}+n-j+1)_{i}}p(n_{0}+n-j+1,t),\ \ n\ge 0.
 	\end{equation*}
 	Equivalently,
 	\begin{align}\label{above}
 	\frac{\mathrm{d}}{\mathrm{d}t}S(m,t)&=-\sum_{j=1}^{k}\sum_{i=j}^{k}\lambda_{(m-j+1)_{i}}p(m-j+1,t)\nonumber\\
 	&=-\sum_{i=1}^{k}\sum_{j=1}^{i}\lambda_{(m-j+1)_{i}}p(m-j+1,t), \ \ m\ge n_{0}.
 	\end{align}
 	Note that  $S(m,0)=1$ for all $m\ge n_{0}$. On integrating (\ref{above}), we get
 	\begin{equation}\label{cauchy apply}
 	1-S(m,t)=\sum_{i=1}^{k}\sum_{j=1}^{i}\lambda_{(m-j+1)_{i}}\int_{0}^{t}p(m-j+1,s)\mathrm{d}s.
 	\end{equation}
 	Let $\eta(t)=\lim\limits_{m\to \infty}(1-S(m,t))$. Using the monotonicity of $S(m,t)$ and applying the Cauchy-Schwarz inequality in (\ref{cauchy apply}) successively, we get
 	\begin{align*}
 	0\le \eta(t)&\le\sum_{i=1}^{k}\left(\sum_{j=1}^{i}\lambda^{2}_{(m-j+1)_{i}}\right)^{1/2}\left(\sum_{j=1}^{i}\left(\int_{0}^{t}p(m-j+1,s)\mathrm{d}s\right)^{2}\right)^{1/2}\\
 	&\le \left(\sum_{i=1}^{k}\sum_{j=1}^{i}\lambda^{2}_{(m-j+1)_{i}}\right)^{1/2}\left(\sum_{i=1}^{k}\sum_{j=1}^{i}\left(\int_{0}^{t}p(m-j+1,s)\mathrm{d}s\right)^{2}\right)^{1/2}
 	\\
 	&\le \left(\sum_{i=1}^{k}\sum_{j=1}^{i}\lambda^{2}_{(m-j+1)_{i}}\right)^{1/2}\sum_{i=1}^{k}\sum_{j=1}^{i}\int_{0}^{t}p(m-j+1,s)\mathrm{d}s. 
 	\end{align*}
 	Thus,
 	\begin{align*}
 	\sum_{m=n_{0}}^{n_{0}+n}\left(\sum_{i=1}^{k}\sum_{j=1}^{i}\lambda^{2}_{(m-j+1)_{i}}\right)^{-1/2}\eta(t)&\le\sum_{i=1}^{k}\sum_{j=1}^{i}\int_{0}^{t}\sum_{m=n_{0}}^{n_{0}+n}p(m-j+1,s)\mathrm{d}s\\
 	&\le k^{2}\int_{0}^{t}S(n_{0}+n,s)\mathrm{d}s.
 	\end{align*}
 	If $\sum_{m=n_{0}}^{\infty}\left(\sum_{i=1}^{k}\sum_{j=1}^{i}\lambda^{2}_{(m-j+1)_{i}}\right)^{-1/2}=\infty$ then $\eta(t)=0$ for all $t$. This implies $S(n_{0}+n,t)\to 1$ as $n\to \infty$, that is, $\sum_{m=n_{0}}^{\infty}p(m,t)=1$.
 \end{proof}
 \begin{remark}
 	On taking $n_{0}=k=1$ and $\lambda_{(n)_{1}}$'s distinct for all $n\ge1$ in Theorem \ref{explosion}, we get the no explosion condition, that is, $\sum_{n=1}^{\infty}\lambda_{(n)_{1}}^{-1}=\infty$	for the pure birth process (see Feller (1968), p. 452).
 \end{remark}

Let $\tilde{p}^{\alpha}(n,s)=\displaystyle \int_{0}^{\infty}e^{-st}{p}^{\alpha}(n,t)\mathrm{d}t$, $s>0$ denote the Laplace transform of the state probabilities of GFBP. In Proposition \ref{firstp}, we obtain the explicit expressions for $\tilde{p}^{\alpha}(n,s)$, $n\ge n_{0}$.

 First, we set some notations.

Let  $I_{n\times n}$ denote an identity matrix,  $O_{n\times m}$ be a zero matrix and $\bar{e}^{j}_{n}$ be a $n$tuple unit row vector with unity at $j$th place and $0$ elsewhere. For $m>n$, we define a matrix $Q$ of order $n\times m$ as $Q_{n\times m}=\begin{bmatrix}
I_{n\times n}&O_{n\times (m-n)}\end{bmatrix}$. For any positive integer $k$, let $\Theta_{n}^{k}\subset \mathbb{R}^{n}$ be defined  as follows: $\Theta_{1}^{k}=\{(1)\}$, $k\ge 1$ and 
\begin{equation}\label{theta}
\Theta_{n}^{k}=\begin{cases*}
\overset{k}{\underset{i=1}{\cup}}\{\Theta_{n-i}^{k} Q_{(n-i)\times n}+i\bar{e}^{n-i+1}_{n}\},\ \ 1\le k \le n-1,\\
\overset{n-1}{\underset{i=1}{\cup}}\{\Theta_{n-i}^{n} Q_{(n-i)\times n} +i\bar{e}^{n-i+1}_{n}\}\cup \{n\bar{e}^{1}_{n}\},\ \ k\ge n,
\end{cases*}
\end{equation}
where $\Theta_{n-i}^{k} Q_{(n-i)\times n}\coloneqq\{\bar{x}Q_{(n-i)\times n}:\bar{x}=(x_{1},x_{2},\dots,x_{n-i})\in \Theta_{n-i}^{k}\}$. 

Using (\ref{theta}), we explicitly write the set $\Theta_{n}^{k}$ for $n=2,3,4$ as follows:
\begin{align*}
\Theta_{2}^{k}&=\begin{cases*}
\{(1,1)\},\ \ k=1,\\
\{(1,1), (2,0)\},\ \ k\ge 2,
\end{cases*}\\[7pt]
\Theta_{3}^{k}&=\begin{cases*}
\{(1,1,1)\},\ \ k=1,\\
\{(1,1,1), (1,2,0), (2,0,1)\},\ \ k=2,\\
\{(1,1,1), (1,2,0), (2,0,1), (3,0,0)\},\ \ k\ge 3,
\end{cases*}\\[7pt]
\Theta_{4}^{k}&=\begin{cases*}
\{(1,1,1,1)\},\ \ k=1,\\
\{(1,1,1,1), (1,2,0,1), (1,1,2,0), (2,0,1,1), (2,0,2,0)\},\ \ k=2,\\
\{(1,1,1,1), (1,2,0,1), (1,1,2,0),\\
\ \ (1,3,0,0), (2,0,1,1), (2,0,2,0), (3,0,0,1)\},\ \ k=3,\\
\{(1,1,1,1), (1,2,0,1), (1,1,2,0), (1,3,0,0),\\
\ \ (2,0,1,1), (2,0,2,0), (3,0,0,1), (4,0,0,0)\},\ \ k\ge4.
\end{cases*}
\end{align*}

For $n\ge 1$, let $\Lambda_{n}=\mathbb{N}^{n}_{0}\setminus\{1\le j\le n-1:x_{j+1}=0\}$ where $(x_{1},x_{2},\dots,x_{n})\in \Theta_{n}^{k}$ and $\mathbb{N}^{n}_{0}=\{0,1,\dots,n\}$. Let $n^{*}=|\Lambda_{n}|$ and 
\begin{equation}\label{(j)}
(j)_{l}=\begin{cases*}
0,\ \ l=1,\\
\inf\{(j)_{l-1}<j\leq n^{*}:x_{j+1}\neq 0\},\ \ 1<l<n^{*},\\
n,\  \ l=n^{*}.
\end{cases*}
\end{equation}

\begin{proposition}\label{firstp}
 For $n=n_0$, the Laplace transform of the state probability of GFBP is given by
	\begin{equation}\label{bgsda22}
\tilde{p}^{\alpha}(n_{0},s)=	\dfrac{s^{\alpha-1}}{s^{\alpha}+\sum_{i=1}^{k}\lambda_{(n_{0})_{i}}}
	\end{equation}
	and for $n=n_0+m$, $m\ge 1$, it is given by
	\begin{equation}\label{laplacem}
	\tilde{p}^{\alpha}(n_{0}+m,s)=\begin{cases*}
	\displaystyle\sum_{\Theta_{m}^{k}}\frac{s^{\alpha-1}\prod_{j=0}^{m-1}\lambda_{(n_{0}+j)_{x_{j+1}}}}{\underset{j\in \Lambda_{m}}{\prod}\left(s^{\alpha}+\sum_{i=1}^{k}\lambda_{(n_{0}+j)_{i}}\right)},\ \ 1\le k\le m-1,\\
	\displaystyle\sum_{\Theta_{m}^{m}}\frac{s^{\alpha-1}\prod_{j=0}^{m-1}\lambda_{(n_{0}+j)_{x_{j+1}}}}{\underset{j\in \Lambda_{m}}{\prod}\left(s^{\alpha}+\sum_{i=1}^{k}\lambda_{(n_{0}+j)_{i}}\right)},\ \  k\ge m,
	\end{cases*}
	\end{equation}
	where $(x_{1},x_{2},\dots,x_{m})\in \Theta_{m}^{k}$ and $\lambda_{(n)_0}$ is interpreted as unity for all $n\geq n_0$.
	\end{proposition}
\begin{proof}
On substituting $n=n_{0}$ in (\ref{model5}), we get
\begin{equation}\label{gfd}
\frac{\mathrm{d}^{\alpha}}{\mathrm{d}t^{\alpha}}p^{\alpha}(n_{0},t)=-\sum_{i=1}^{k}\lambda_{(n_{0})_{i}} p^{\alpha}(n_{0},t),
\end{equation}
where we have used the fact that  $p^{\alpha}(m,t)=0$ for all $m<n_{0}$. On taking the Laplace transform in (\ref{gfd}) and using (\ref{laplace}), we obtain
\begin{equation*}
\tilde{p}^{\alpha}(n_{0},s)=\frac{s^{\alpha-1}}{s^{\alpha}+\sum_{i=1}^{k}\lambda_{(n_{0})_{i}}}.
\end{equation*}
On taking $n=n_{0}+1$ in (\ref{model5}), we get
\begin{equation*}
\frac{\mathrm{d}^{\alpha}}{\mathrm{d}t^{\alpha}}p^{\alpha}(n_{0}+1,t)=-\sum_{i=1}^{k}\lambda_{(n_{0}+1)_{i}} p^{\alpha}(n_{0}+1,t)+\lambda_{(n_{0})_{1}}p^{\alpha}(n_{0},t),
\end{equation*}
which on taking Laplace transform gives
\begin{align*}
\tilde{p}^{\alpha}(n_{0}+1,s)&=\frac{\lambda_{(n_{0})_{1}}\tilde{p}^{\alpha}(n_{0},s)}{s^{\alpha}+\sum_{i=1}^{k}\lambda_{(n_{0}+1)_{i}}}\\
&=\frac{\lambda_{(n_{0})_{1}}s^{\alpha-1}}{\left(s^{\alpha}+\sum_{i=1}^{k}\lambda_{(n_{0})_{i}}\right)\left(s^{\alpha}+\sum_{i=1}^{k}\lambda_{(n_{0}+1)_{i}}\right)}.
\end{align*}
For $n=n_{0}+2$ in (\ref{model5}), we get
\begin{equation*}
\frac{\mathrm{d}^{\alpha}}{\mathrm{d}t^{\alpha}}p^{\alpha}(n_{0}+2,t)=-\sum_{i=1}^{k}\lambda_{(n_{0}+2)_{i}} p^{\alpha}(n_{0}+2,t)+\sum_{i=1}^{k}\lambda_{(n_{0}+2-i)_{i}}p^{\alpha}(n_{0}+2-i,t),
\end{equation*}
which gives
\begin{equation*}
\tilde{p}^{\alpha}(n_{0}+2,s)=\frac{\sum_{i=1}^{k}\lambda_{(n_{0}+2-i)_{i}}\tilde{p}^{\alpha}(n_{0}+2-i,s)}{s^{\alpha}+\sum_{i=1}^{k}\lambda_{(n_{0}+2)_{i}}}.
\end{equation*}
Thus,
\begin{equation*}
\tilde{p}^{\alpha}(n_{0}+2,s)=\begin{cases*}
\displaystyle\frac{\lambda_{(n_{0})_{1}}\lambda_{(n_{0}+1)_{1}}s^{\alpha-1}}{\overset{2}{\underset{j=0}{\prod}}\left(s^{\alpha}+\sum_{i=1}^{k}\lambda_{(n_{0}+j)_{i}}\right)}, \ k=1,\\
\displaystyle\frac{\lambda_{(n_{0})_{1}}\lambda_{(n_{0}+1)_{1}}s^{\alpha-1}}{\overset{2}{\underset{j=0}{\prod}}\left(s^{\alpha}+\sum_{i=1}^{k}\lambda_{(n_{0}+j)_{i}}\right)}+\frac{\lambda_{(n_{0})_{2}}s^{\alpha-1}}{\underset{j\neq1}{\overset{2}{\underset{j=0}{\prod}}}\left(s^{\alpha}+\sum_{i=1}^{k}\lambda_{(n_{0}+j)_{i}}\right)}, \ k\ge2.
\end{cases*}
\end{equation*}
Similarly, we get the following expressions in the case of $n=n_{0}+3$: 
\begin{equation*}
\tilde{p}^{\alpha}(n_{0}+3,s)=\frac{\sum_{i=1}^{k}\lambda_{(n_{0}+3-i)_{i}}\tilde{p}^{\alpha}(n_{0}+3-i,s)}{s^{\alpha}+\sum_{i=1}^{k}\lambda_{(n_{0}+3)_{i}}}.
\end{equation*}
For $k=1$, we get
\begin{equation*}
\tilde{p}^{\alpha}(n_{0}+3,s)=\frac{\lambda_{(n_{0})_{1}}\lambda_{(n_{0}+1)_{1}}\lambda_{(n_{0}+2)_{1}}s^{\alpha-1}}{\overset{3}{\underset{j=0}{\prod}}\left(s^{\alpha}+\sum_{i=1}^{k}\lambda_{(n_{0}+j)_{i}}\right)}.
\end{equation*}
For $k=2$, we get
\begin{align*}
\tilde{p}^{\alpha}(n_{0}+3,s)&=\frac{\lambda_{(n_{0})_{1}}\lambda_{(n_{0}+1)_{1}}\lambda_{(n_{0}+2)_{1}}s^{\alpha-1}}{\overset{3}{\underset{j=0}{\prod}}\left(s^{\alpha}+\sum_{i=1}^{k}\lambda_{(n_{0}+j)_{i}}\right)}+\frac{\lambda_{(n_{0})_{2}}\lambda_{(n_{0}+2)_{1}}s^{\alpha-1}}{\underset{j\neq1}{\overset{3}{\underset{j=0}{\prod}}}\left(s^{\alpha}+\sum_{i=1}^{k}\lambda_{(n_{0}+j)_{i}}\right)}\\
&\ \ +\frac{\lambda_{(n_{0})_{1}}\lambda_{(n_{0}+1)_{2}}s^{\alpha-1}}{\underset{j\neq2}{\overset{3}{\underset{j=0}{\prod}}}\left(s^{\alpha}+\sum_{i=1}^{k}\lambda_{(n_{0}+j)_{i}}\right)}.
\end{align*}
For $k\ge3$, we get
\begin{align*}
\tilde{p}^{\alpha}(n_{0}+3,s)&=\frac{\lambda_{(n_{0})_{1}}\lambda_{(n_{0}+1)_{1}}\lambda_{(n_{0}+2)_{1}}s^{\alpha-1}}{\overset{3}{\underset{j=0}{\prod}}\left(s^{\alpha}+\sum_{i=1}^{k}\lambda_{(n_{0}+j)_{i}}\right)}+\frac{\lambda_{(n_{0})_{2}}\lambda_{(n_{0}+2)_{1}}s^{\alpha-1}}{\underset{j\neq1}{\overset{3}{\underset{j=0}{\prod}}}\left(s^{\alpha}+\sum_{i=1}^{k}\lambda_{(n_{0}+j)_{i}}\right)}\\
&\ \ +\frac{\lambda_{(n_{0})_{1}}\lambda_{(n_{0}+1)_{2}}s^{\alpha-1}}{\underset{j\neq2}{\overset{3}{\underset{j=0}{\prod}}}\left(s^{\alpha}+\sum_{i=1}^{k}\lambda_{(n_{0}+j)_{i}}\right)}+\frac{\lambda_{(n_{0})_{3}}s^{\alpha-1}}{\underset{j\neq1,2}{\overset{3}{\underset{j=0}{\prod}}}\left(s^{\alpha}+\sum_{i=1}^{k}\lambda_{(n_{0}+j)_{i}}\right)}.
\end{align*}
Thus,
\begin{equation*}
\tilde{p}^{\alpha}(n_{0}+3,s)=\begin{cases*}
\displaystyle\sum_{\Theta_{3}^{k}}\frac{s^{\alpha-1}\prod_{j=0}^{2}\lambda_{(n_{0}+j)_{x_{j+1}}}}{\underset{j\in \Lambda_{3}}{\prod}\left(s^{\alpha}+\sum_{i=1}^{k}\lambda_{(n_{0}+j)_{i}}\right)},\ \ 1\le k\le 2,\\
\displaystyle\sum_{\Theta_{3}^{3}}\frac{s^{\alpha-1}\prod_{j=0}^{2}\lambda_{(n_{0}+j)_{x_{j+1}}}}{\underset{j\in \Lambda_{3}}{\prod}\left(s^{\alpha}+\sum_{i=1}^{k}\lambda_{(n_{0}+j)_{i}}\right)},\ \  k\ge 3.
\end{cases*}
\end{equation*}
By iterating this procedure, we get the required result.
\end{proof}

The next result follows by taking $\nu_{n}=\alpha$ for all $n\ge 1$ in Eq. (9) and Eq. (39) of Kataria and Vellaisamy (2019).
\begin{proposition}\label{prop3}
	Let $\lambda_{j}$, $j=1,2,\dots,n$ be any positive real numbers. Then,
	\begin{equation*}
	\mathcal{L}^{-1}\left(\frac{s^{\alpha-1}}{\prod_{j=1}^{n}(s^{\alpha}+\lambda_{j})};t\right)=\frac{(-1)^{n-1}}{\overset{n}{\underset{j=2}{\prod}}\lambda_{j}}\sum_{i=n-1}^{\infty}\frac{(-t^{\alpha })^{i}}{\Gamma(i\alpha +1)}\sum_{\Omega^{i}_{n}}\prod_{j=1}^{n}\lambda_{j}^{y_{j}},
	\end{equation*}	
	where $\Omega^{i}_{n}=\{(y_{1},y_{2},\dots,y_{n}):\sum_{j=1}^{n}y_{j}=i, \ y_{1}\in \mathbb{N}_{0}, \ y_{j}\in \mathbb{N}_{0}\setminus\{0\}, \ 2\le j \le n\}$.
\end{proposition}
In case $\lambda_{j}$'s are distinct, the above result has a simplified form. The following result is proved using method of induction by Alipour {\it et al.} (2015). Here, we give an alternate proof.
\begin{lemma}\label{lemma}
Let \begin{equation*}
p(x)=\sum_{i=1}^{n}\underset{j\neq i}{\overset{n}{\underset{j=1}{\prod}}}\frac{x+\lambda_{j}}{\lambda_{j}-\lambda_{i}}-1, \ \ n>1, 
\end{equation*}
where $x\in \mathbb{R}$ and  $\lambda_{1},\lambda_{2},\dots,\lambda_{n}$ be distinct	real numbers. Then, $p(x)\equiv0$.
\end{lemma}
\begin{proof}
Note that $p(-\lambda_{j})=0$ for $j=1,2,\dots,n$, that is, $p(x)$ is a polynomial of degree $n-1$ with $n$ distinct roots. This implies that it is a zero polynomial. This proves the result.	
\end{proof}	

The following result will be used (see Kilbas {\it et al.} (2006)):
\begin{equation}\label{mllaplace}
\int_{0}^{\infty}e^{-st}E_{\alpha,1}(-\omega t^{\alpha})\,\mathrm{d}t=\frac{s^{\alpha-1}}{s^{\alpha}+\omega},
\end{equation}
where $E_{\alpha,1}(\cdot)$ is the Mittag-Leffler function defined as
\begin{equation*}
E_{\alpha,1}(x)\coloneqq\sum_{j=0}^{\infty} \frac{x^{j}}{\Gamma(j\alpha+1)},\ \ x\in\mathbb{R},\ \ \alpha>0.
\end{equation*}
\begin{proposition}\label{prop2}
Let $\lambda_{1},\lambda_{2},\dots,\lambda_{n}$ be distinct	positive real numbers. Then,
\begin{equation*}
\mathcal{L}^{-1}\left(\frac{s^{\alpha-1}}{\prod_{j=1}^{n}(s^{\alpha}+\lambda_{j})};t\right)=\sum_{i=1}^{n}\frac{E_{\alpha,1}(-\lambda_{i}t^{\alpha})}{\overset{n}{\underset{j=1, j\neq i}{\prod}}(\lambda_{j}-\lambda_{i})},\ \ n>1.
\end{equation*}	

\end{proposition}
\begin{proof}
Using Lemma \ref{lemma}, we get
\begin{align*}
\mathcal{L}^{-1}\left(\frac{s^{\alpha-1}}{\prod_{j=1}^{n}(s^{\alpha}+\lambda_{j})};t\right)&=\mathcal{L}^{-1}\left(\sum_{i=1}^{n}\frac{s^{\alpha-1}}{s^{\alpha}+\lambda_{i}}\underset{j\neq i}{\overset{n}{\underset{j=1}{\prod}}}\frac{1}{(\lambda_{j}-\lambda_{i})};t\right)\\
&=\sum_{i=1}^{n}\frac{E_{\alpha,1}(-\lambda_{i}t^{\alpha})}{\overset{n}{\underset{j=1, j\neq i}{\prod}}(\lambda_{j}-\lambda_{i})},
\end{align*}
where the last step follows from (\ref{mllaplace}). 
\end{proof}

\begin{theorem}\label{main}
	For $n=n_{0}$, the state probability of GFBP is given by
	 \begin{equation*}
	 p^{\alpha}(n_{0},t)=E_{\alpha,1}\left(-\sum_{i=1}^{k}\lambda_{(n_{0})_{i}}t^{\alpha}\right)
	 \end{equation*}
and for $n=n_0+m$, $m\ge 1$, it is given by
	\begin{equation}\label{pnta}
	p^{\alpha}(n_{0}+m,t)=\begin{cases*}
	\displaystyle\sum_{\Theta_{m}^{k}}\prod_{j=0}^{m-1}\lambda_{(n_{0}+j)_{x_{j+1}}}\frac{(-1)^{m^{*}-1}}{\overset{m^{*}}{\underset{l=2}{\prod}}\mu_{l}}\sum_{i=m^{*}-1}^{\infty}\frac{(-t^{\alpha })^{i}}{\Gamma(i\alpha +1)}\sum_{\Omega^{i}_{m^{*}}}\prod_{l=1}^{m^{*}}\mu_{l}^{y_{l}},\ \ 1\le k\le m-1,\\
	\displaystyle\sum_{\Theta_{m}^{m}}\prod_{j=0}^{m-1}\lambda_{(n_{0}+j)_{x_{j+1}}}\frac{(-1)^{m^{*}-1}}{\overset{m^{*}}{\underset{l=2}{\prod}}\mu_{l}}\sum_{i=m^{*}-1}^{\infty}\frac{(-t^{\alpha })^{i}}{\Gamma(i\alpha +1)}\sum_{\Omega^{i}_{m^{*}}}\prod_{l=1}^{m^{*}}\mu_{l}^{y_{l}},\ \  k\ge m,
	\end{cases*}
	\end{equation}
	where $(x_{1},x_{2},\dots,x_{m})\in \Theta_{m}^{k}$, $(y_{1},y_{2},\dots,y_{m^{*}})\in \Omega^{i}_{m^{*}}$, $m^{*}=|\Lambda_{m}|$ and $\mu_{l}=\sum_{i=1}^{k}\lambda_{(n_{0}+(j)_{l})_{i}}$,  $(j)_{l}$'s are given by (\ref{(j)}). 
	
	If $\mu_{l}$, $l=1,2,\dots,m^{*}$ are distinct then
	\begin{equation}\label{pnt}
	p^{\alpha}(n_{0}+m,t)=\begin{cases*}
	\displaystyle\sum_{\Theta_{m}^{k}}\prod_{j=0}^{m-1}\lambda_{(n_{0}+j)_{x_{j+1}}}\sum_{i=1}^{m^{*}}\frac{E_{\alpha,1}(-\mu_{i}t^{\alpha})}{\overset{m^{*}}{\underset{l=1, l\neq i}{\prod}}(\mu_{l}-\mu_{i})},\ \ 1\le k\le m-1,\\
	\displaystyle\sum_{\Theta_{m}^{m}}\prod_{j=0}^{m-1}\lambda_{(n_{0}+j)_{x_{j+1}}}\sum_{i=1}^{m^{*}}\frac{E_{\alpha,1}(-\mu_{i}t^{\alpha})}{\overset{m^{*}}{\underset{l=1, l\neq i}{\prod}}(\mu_{l}-\mu_{i})},\ \  k\ge m.
	\end{cases*}
	\end{equation}
\end{theorem}
\begin{proof}
	For the case $n=n_{0}$, the result follows by taking inverse Laplace transform in (\ref{bgsda22}) and then by using (\ref{mllaplace}). For $n>n_0$, the Laplace transform (\ref{laplacem}) of GFBP can be rewritten as
\begin{equation*}
\tilde{p}^{\alpha}(n_{0}+m,s)=\begin{cases*}
\displaystyle\sum_{\Theta_{m}^{k}}\prod_{j=0}^{m-1}\lambda_{(n_{0}+j)_{x_{j+1}}}\frac{s^{\alpha-1}}{\overset{m^{*}}{\underset{l=1}{\prod}}\left(s^{\alpha}+\mu_{l}\right)},\ \ 1\le k\le m-1,\\
\displaystyle\sum_{\Theta_{m}^{m}}\prod_{j=0}^{m-1}\lambda_{(n_{0}+j)_{x_{j+1}}}\frac{s^{\alpha-1}}{\overset{m^{*}}{\underset{l=1}{\prod}}\left(s^{\alpha}+\mu_{l}\right)},\ \  k\ge m.
\end{cases*}
\end{equation*}

The result follows on taking inverse Laplace transform in the above equation and using Proposition \ref{prop3} and Proposition \ref{prop2}, respectively.
\end{proof}
\begin{remark}
	Let $W_{1}^{\alpha}$ denote the first waiting time of GFBP. Its distribution is given by
	\begin{equation*}
	\mathrm{Pr}\{W_{1}^{\alpha}>t\}=\mathrm{Pr}\{\mathcal{X}^{\alpha}(t)=n_{0}\}=E_{\alpha,1}\left(-\sum_{i=1}^{k}\lambda_{(n_{0})_{i}}t^{\alpha}\right).
	\end{equation*}
\end{remark}
\begin{remark}
Recall that the TFPP is obtained as a particular case of the GFBP
when $n_{0}=0$, $k=1$ and $\lambda_{(n)_{1}}=\lambda$ for all $n\ge0$. In this case, we have $\Theta_{m}^{1}=\{(1,1,\dots,1)\}$, $m^{*}=|\mathbb{N}^{m}_{0}|=m+1$, $|\Omega^{i}_{m+1}|=\frac{i!}{m!(i-m)!}$ and $\mu_{l}=\lambda_{((j)_{l})_{1}}=\lambda$ as  $(j)_{l}=l-1$, for all $1\le l \le m+1$. On substituting these values in (\ref{pnta}), we get
	\begin{align*}
	p^{\alpha}(m,t)&=\sum_{\Theta_{m}^{1}}\prod_{j=0}^{m-1}\lambda_{(j)_{1}}\frac{(-1)^{m}}{\overset{m+1}{\underset{l=2}{\prod}}\mu_{l}}\sum_{i=m}^{\infty}\frac{(-t^{\alpha })^{i}}{\Gamma(i\alpha +1)}\sum_{\Omega^{i}_{m+1}}\prod_{l=1}^{m+1}\mu_{l}^{y_{l}}\\
	&=(-1)^{m}\sum_{i=m}^{\infty}\frac{(-\lambda t^{\alpha })^{i}}{\Gamma(i\alpha +1)}\sum_{\Omega^{i}_{m+1}}1\\
	&=(-1)^{m}\sum_{i=m}^{\infty}\frac{(-\lambda t^{\alpha })^{i}}{\Gamma(i\alpha +1)}\frac{i!}{m!(i-m)!}\\
	&=\frac{(\lambda t^{\alpha})^{m}}{m!}\sum_{i=0}^{\infty}\frac{(i+m)!}{i!}\frac{(-\lambda t^{\alpha })^{i}}{\Gamma((i+m)\alpha +1)},
	\end{align*}
	which agrees with the probability mass function (pmf) of TFPP (see Beghin and Orsingher (2009), Eq. (2.10)).
\end{remark}
\begin{remark}
Recall that the FPBP is obtained as a particular case of the GFBP
when $n_{0}=k=1$ and $\lambda_{(n)_{1}}$'s are distinct. For notational convenience, we write $\lambda_{(n)_{1}}=\lambda_{n}$ for all $n\ge1$. In this case, we have $\Theta_{m}^{1}=\{(1,1,\dots,1)\}$, $m^{*}=|\mathbb{N}^{m}_{0}|=m+1$  and $\mu_{l}=\lambda_{(1+(j)_{l})_{1}}=\lambda_{l}$ as  $(j)_{l}=l-1$, for all $1\le l \le m+1$.  On substituting these values in (\ref{pnt}), we get
		\begin{align*}
	p^{\alpha}(m+1,t)&=\sum_{\Theta_{m}^{1}}\prod_{j=0}^{m-1}\lambda_{j+1}\sum_{i=1}^{m+1}\frac{E_{\alpha,1}(-\lambda_{i}t^{\alpha})}{\overset{m+1}{\underset{l=1, l\neq i}{\prod}}(\lambda_{l}-\lambda_{i})}\\
	&=\prod_{j=1}^{m}\lambda_{j}\sum_{i=1}^{m+1}\frac{E_{\alpha,1}(-\lambda_{i}t^{\alpha})}{\overset{m+1}{\underset{l=1, l\neq i}{\prod}}(\lambda_{l}-\lambda_{i})},
		\end{align*}
		which agrees with the pmf of FPBP (see Orsingher and Polito (2010), Eq. (2.3)).
\end{remark}

Let $f_{T_{2\alpha}}(x,t)$ be the folded solution of the following fractional diffusion equation: 
\begin{equation}\label{diff}
\frac{\mathrm{d}^{2\alpha}}{\mathrm{d}t^{2\alpha}}u(x,t)=\frac{\partial^{2}}{\partial x^{2}}u(x,t),\ \ x\in\mathbb{R},\ t>0,
\end{equation}
with $u(x,0)=\delta(x)$ for $0<\alpha\le 1$ and $\frac{\partial^{}}{\partial t}u(x,0)=0$ for $1/2<\alpha\le1$.

Its Laplace transform is given by (see Orsingher and Polito (2010), Eq. (2.29))
\begin{equation}\label{lta}
\int_{0}^{\infty}e^{-st}f_{T_{2\alpha}}(x,t)\mathrm{d}t=s^{\alpha-1}e^{-x s^{\alpha}}, \ \ x>0.
\end{equation}
Let $\{T_{2\alpha}(t)\}_{t>0}$ denote a random process whose one-dimensional distribution is given by $f_{T_{2\alpha}}(x,t)$. 	

\begin{theorem}\label{subordination}
	Let $\mu_{l}$'s are as defined in Theorem \ref{main} and  $\{T_{2\alpha}(t)\}_{t>0}$, $0<\alpha\le 1$, be a process whose distribution solves (\ref{diff}). If $\mu_{l}$'s are distinct then the following holds for GFBP: 
	\begin{equation}\label{sub}
	\mathcal{X}^{\alpha}(t)\overset{d}{=}\mathcal{X}(T_{2\alpha}(t)),
	\end{equation}
	where $\{T_{2\alpha}(t)\}_{t>0}$ is  independent of the GBP $\{\mathcal{X}(t)\}_{t>0}$.
\end{theorem}
\begin{proof}
Let $\mu_{0}=\sum_{i=1}^{k}\lambda_{(n_{0})_{i}}$. The Laplace transform of the probability generating function of GFBP can be written as
	\begin{align*}
	\tilde{G}^{\alpha}(u,s)&=\int_{0}^{\infty}e^{-st}\sum_{n=n_{0}}^{\infty}u^{n}p^{\alpha}(n,t)\mathrm{d}t\\
	&=\int_{0}^{\infty}e^{-st}\left(u^{n_{0}}p^{\alpha}(n_{0},t)+\sum_{n=n_{0}+1}^{n_{0}+k}u^{n}p^{\alpha}(n,t)+\sum_{n=n_{0}+k+1}^{\infty}u^{n}p^{\alpha}(n,t)\right)\mathrm{d}t\\
	&=u^{n_{0}}\int_{0}^{\infty}e^{-st}\left(p^{\alpha}(n_{0},t)+\sum_{m=1}^{k}u^{m}p^{\alpha}(n_{0}+m,t)+\sum_{m=k+1}^{\infty}u^{m}p^{\alpha}(n_{0}+m,t)\right)\mathrm{d}t\\
	&=u^{n_{0}}\Bigg(\frac{s^{\alpha-1}}{s^{\alpha}+\mu_{0}}+\sum_{m=1}^{k}u^{m}\sum_{\Theta_{m}^{m}}\prod_{j=0}^{m-1}\lambda_{(n_{0}+j)_{x_{j+1}}}\sum_{i=1}^{m^{*}}\frac{s^{\alpha-1}}{\overset{m^{*}}{\underset{l=1, l\neq i}{\prod}}(\mu_{l}-\mu_{i})(s^{\alpha}+\mu_{i})}\\
	&\ \ +\sum_{m=k+1}^{\infty}u^{m}\sum_{\Theta_{m}^{k}}\prod_{j=0}^{m-1}\lambda_{(n_{0}+j)_{x_{j+1}}}\sum_{i=1}^{m^{*}}\frac{s^{\alpha-1}}{\overset{m^{*}}{\underset{l=1, l\neq i}{\prod}}(\mu_{l}-\mu_{i})(s^{\alpha}+\mu_{i})}\Bigg),\ \  \text{(using\ (\ref{mllaplace}) and (\ref{pnt}))}\\
	&=s^{\alpha-1}\int_{0}^{\infty}e^{-\xi s^{\alpha}}\Bigg(u^{n_{0}}e^{-\xi \mu_{0}}+\sum_{m=1}^{k}u^{m+n_{0}}\sum_{\Theta_{m}^{m}}\prod_{j=0}^{m-1}\lambda_{(n_{0}+j)_{x_{j+1}}}\sum_{i=1}^{m^{*}}\frac{e^{-\xi \mu_{i}}}{\overset{m^{*}}{\underset{l=1, l\neq i}{\prod}}(\mu_{l}-\mu_{i})}\\
	&\ \ +\sum_{m=k+1}^{\infty}u^{m+n_{0}}\sum_{\Theta_{m}^{k}}\prod_{j=0}^{m-1}\lambda_{(n_{0}+j)_{x_{j+1}}}\sum_{i=1}^{m^{*}}\frac{e^{-\xi \mu_{i}}}{\overset{m^{*}}{\underset{l=1, l\neq i}{\prod}}(\mu_{l}-\mu_{i})}\Bigg)\mathrm{d}\xi\\
	&=s^{\alpha-1}\int_{0}^{\infty}e^{-\xi s^{\alpha}}\left(u^{n_{0}}p(n_{0},\xi)+\sum_{m=1}^{k}u^{m+n_{0}}p(n_{0}+m,\xi)+\sum_{m=k+1}^{\infty}u^{m+n_{0}}p(n_{0}+m,\xi)\right)\mathrm{d}\xi\\
	&=\int_{0}^{\infty}s^{\alpha-1}e^{-\xi s^{\alpha}}G(u,\xi)\mathrm{d}\xi\\
	&=\int_{0}^{\infty}G(u,\xi)\int_{0}^{\infty}e^{-st}f_{T_{2\alpha}}(\xi,t)\mathrm{d}t \mathrm{d}\xi,\ \  \text{(using\ (\ref{lta}))}\\
	&=\int_{0}^{\infty}e^{-st}\left(\int_{0}^{\infty}G(u,\xi)f_{T_{2\alpha}}(\xi,t)\mathrm{d}\xi\right)\mathrm{d}t .
	\end{align*}
	By uniqueness of Laplace transform, we get
	\begin{equation*}
	G^{\alpha}(u,t)=\int_{0}^{\infty}G(u,\xi)f_{T_{2\alpha}}(\xi,t)\mathrm{d}\xi.
	\end{equation*}	
	This completes the proof.
\end{proof}
\begin{remark}
	For $\alpha=1/2$, the process $\{T_{2\alpha}(t)\}_{t>0}$ becomes a reflecting Brownian motion $\{|B(t)|\}_{t>0}$ (see Beghin and Orsingher (2009)) as the diffusion equation (\ref{diff}) reduces to the heat equation
	\begin{equation*}
	\begin{cases*}
	\frac{\partial}{\partial t}u(x,t)=\frac{\partial^{2}}{\partial x^{2}}u(x,t),\ \ x\in\mathbb{R},\ t>0,\\
	u(x,0)=\delta(x).
	\end{cases*}
	\end{equation*}
	So, $\mathcal{X}^{1/2}(t)$ coincides with GBP at a Brownian time, that is, $\mathcal{X}^{1/2}(t)\overset{d}{=}\mathcal{X}(|B(t)|)$, $t>0$.
\end{remark}

\begin{remark}
	Using (\ref{sub}), we get the following relationship between the pmfs of GFBP and GBP:
	\begin{equation*}
p^{\alpha}(m,t)=\int_{0}^{\infty}p(m,\xi)f_{T_{2\alpha}}(\xi,t)\mathrm{d}\xi.
	\end{equation*}
It implies that $\sum_{m=n_{0}}^{\infty}p^{\alpha}(m,t)=1$ if and only if $\sum_{m=n_{0}}^{\infty}p(m,t)=1$. From Theorem \ref{explosion}, it follows that $\sum_{m=n_{0}}^{\infty}\left(\sum_{i=1}^{k}\sum_{j=1}^{i}\lambda^{2}_{(m-j+1)_{i}}\right)^{-1/2}=\infty$ is the non-exploding condition for GFBP.
\end{remark}
\subsection{A limiting case of GFBP}
  If we consider a process $\{\mathcal{X}_{1}^{\alpha}(t)\}_{t\ge0}$ which has the jumps of any size $j\ge 1$ then from the system (\ref{model2}) of GFBP, we get
\begin{equation*}
\frac{\mathrm{d}^{\alpha}}{\mathrm{d}t^{\alpha}}p^{\alpha}_{1}(n,t)=-\sum_{i=1}^{\infty}\lambda_{(n)_{i}} p^{\alpha}_{1}(n,t)+	\sum_{i=1}^{n}\lambda_{(n-i)_{i}}p^{\alpha}_{1}(n-i,t),\ \ n\ge n_{0},
\end{equation*}
where $p^{\alpha}_{1}(n,t)=\mathrm{Pr}\{\mathcal{X}_{1}^{\alpha}(t)=n\}$, $0<\alpha \le 1$. Here, we assume that  $\sum_{i=1}^{\infty}\lambda_{(n)_{i}}<\infty$ for all $n\ge n_{0}$.

From Theorem \ref{main}, we get its state probabilities as  
	$p^{\alpha}_{1}(n_{0},t)=E_{\alpha,1}\left(-\sum_{i=1}^{\infty}\lambda_{(n_{0})_{i}}t^{\alpha}\right)$,
	and  
	\begin{equation*}
	p^{\alpha}_{1}(n_{0}+m,t)=
\sum_{\Theta_{m}^{m}}\prod_{j=0}^{m-1}\lambda_{(n_{0}+j)_{x_{j+1}}}\frac{(-1)^{m^{*}-1}}{\overset{m^{*}}{\underset{l=2}{\prod}}\mu_{l}}\sum_{i=m^{*}-1}^{\infty}\frac{(-t^{\alpha })^{i}}{\Gamma(i\alpha +1)}\sum_{\Omega^{i}_{m^{*}}}\prod_{l=1}^{m^{*}}\mu_{l}^{y_{l}},\ \ m\ge 1,
	\end{equation*}
where $(x_{1},x_{2},\dots,x_{m})\in \Theta_{m}^{m}$, $(y_{1},y_{2},\dots,y_{m^{*}})\in \Omega^{i}_{m^{*}}$, $m^{*}=|\Lambda_{m}|$ and  $\mu_{l}=\sum_{i=1}^{\infty}\lambda_{(n_{0}+(j)_{l})_{i}}$,  $(j)_{l}$'s are given by (\ref{(j)}). If $\mu_{l}$, $l=1,2,\dots,m^{*}$ are distinct then
	\begin{equation*}
	p^{\alpha}_{1}(n_{0}+m,t)=
	\sum_{\Theta_{m}^{m}}\prod_{j=0}^{m-1}\lambda_{(n_{0}+j)_{x_{j+1}}}\sum_{i=1}^{m^{*}}\frac{E_{\alpha,1}(-\mu_{i}t^{\alpha})}{\overset{m^{*}}{\underset{l=1, l\neq i}{\prod}}(\mu_{l}-\mu_{i})}.
	\end{equation*}
	
Let $\{\mathcal{X}_{1}(t)\}_{t\ge 0}$ denote the non-fractional version of $\{\mathcal{X}_{1}^{\alpha}(t)\}_{t>0}$. The following relationship holds:
	\begin{equation*}
	\mathcal{X}^{\alpha}_{1}(t)\overset{d}{=}\mathcal{X}_{1}(T_{2\alpha}(t)),
	\end{equation*}
where $\{T_{2\alpha}(t)\}_{t>0}$ is  independent of $\{\mathcal{X}_{1}(t)\}_{t>0}$. Its proof follows along the similar lines to that of Theorem \ref{subordination}. 


\section{A State dependent version of the GFBP}
Garra {\it et al.} (2015) introduced and studied the state dependent versions of TFPP and FPBP. Here, we introduce a state dependent version of the GFBP. We denote it by $\{\mathcal{Q}(t)\}_{t\ge0}$. It is defined as the stochastic process whose 
state probabilities  $q(n,t)=\mathrm{Pr}\{\mathcal{Q}(t)=n\}$ satisfy the following system of fractional differential equations:
\begin{equation}\label{model3}
\frac{\mathrm{d}^{\alpha_{n}}}{\mathrm{d}t^{\alpha_{n}}}q(n,t)=-\sum_{i=1}^{k}\lambda_{(n)_{i}}q(n,t)+	\sum_{i=1}^{\min\{n,k\}}\lambda_{(n-i)_{i}}q(n-i,t),\ \ 0<\alpha_{n} \le 1,\ n\ge n_{0},
\end{equation}
with the initial conditions
\begin{equation*}
q(n,0)=\begin{cases}
1,\ \ n=n_{0},\\
0,\ \ n>n_{0}.
\end{cases}
\end{equation*}
Note that the System (\ref{model3}) is obtained by taking variable order fractional derivatives in (\ref{model2}). The Caputo fractional derivative of state probability is the convolution of rate of change in state probability and a suitable weight function, that is,
\begin{equation*}
\frac{\mathrm{d}^{\alpha_{n}}}{\mathrm{d}t^{\alpha_{n}}}q(n,t)=\frac{\mathrm{d}}{\mathrm{d}t}q(n,t)*\frac{t^{-\alpha_{n}}}{\Gamma(1-\alpha_{n})}.
\end{equation*}
Here, the power of the weight function depends  on the actual state of the process at time $t\geq0$. Thus, the number of events that have occurred till time $t$ modifies the order of fractional derivative.

 Next, we give the Laplace transform of the state probabilities of state dependent GFBP. Its proof follows similar lines to that of Proposition \ref{firstp}. Let $\Theta_{n}^{k}$, $\Lambda_{n}$ {\it etc.} be as defined in Section \ref{section3}.
\begin{proposition}\label{firstp2}
	For $n=n_0$, the Laplace transform of the state probability of $\{\mathcal{Q}(t)\}_{t\ge0}$ is given by
	\begin{equation*}
	\tilde{q}(n_{0},s)=	\dfrac{s^{\alpha_{n_{0}}-1}}{s^{\alpha_{n_{0}}}+\sum_{i=1}^{k}\lambda_{(n_{0})_{i}}}
	\end{equation*}
	and for $n=n_0+m$, $m\ge 1$, it is given by
	\begin{equation*}
	\tilde{q}(n_{0}+m,s)=\begin{cases*}
	\displaystyle\sum_{\Theta_{m}^{k}}\frac{s^{\alpha_{n_{0}}-1}\prod_{j=0}^{m-1}\lambda_{(n_{0}+j)_{x_{j+1}}}}{\underset{j\in \Lambda_{m}}{\prod}\left(s^{\alpha_{n_{0}+j}}+\sum_{i=1}^{k}\lambda_{(n_{0}+j)_{i}}\right)},\ \ 1\le k\le m-1,\\
	\displaystyle\sum_{\Theta_{m}^{m}}\frac{s^{\alpha_{n_{0}}-1}\prod_{j=0}^{m-1}\lambda_{(n_{0}+j)_{x_{j+1}}}}{\underset{j\in \Lambda_{m}}{\prod}\left(s^{\alpha_{n_{0}+j}}+\sum_{i=1}^{k}\lambda_{(n_{0}+j)_{i}}\right)},\ \  k\ge m,
	\end{cases*}
	\end{equation*}
	where $(x_{1},x_{2},\dots,x_{m})\in \Theta_{m}^{k}$.
\end{proposition}

 The following result holds (see Kataria and Vellaisamy (2019), Eq. (9) and Eq. (39)):
\begin{equation}\label{result}
\mathcal{L}^{-1}\left(\frac{s^{\alpha_{1}-1}}{\prod_{j=1}^{n}(s^{\alpha_{j}}+\lambda_{j})};t\right)=\frac{(-1)^{n-1}}{\overset{n}{\underset{j=2}{\prod}}\lambda_{j}}\sum_{i=n-1}^{\infty}(-1)^i\underset{\Omega^i_n}{\sum}\frac{t^{\sum_{j=1}^ny_j\alpha_j}\prod_{j=1}^n\lambda_j^{y_j}}{\Gamma(1+\sum_{j=1}^ny_j\alpha_j)},\ \ n\geq1,
\end{equation}
where $\lambda_{j}$'s are positive real numbers and  $\Omega^{i}_{n}=\{(y_{1},y_{2},\dots,y_{n}):\sum_{j=1}^{n}y_{j}=i, \ y_{1}\in \mathbb{N}_{0}, \ y_{j}\in \mathbb{N}_{0}\setminus\{0\}, \ 2\le j \le n\}$.

The state probabilities $q(n,t)$ for all $n\geq n_0$ can be obtained by taking inverse Laplace transform in Proposition \ref{firstp2} and by using (\ref{result}).
\begin{theorem}\label{main3}
	For $n=n_{0}$, the state probability of $\{\mathcal{Q}(t)\}_{t\ge0}$ is given by
	$q(n_{0},t)=E_{\alpha_{n_{0}},1}\left(-\sum_{i=1}^{k}\lambda_{(n_{0})_{i}}t^{\alpha_{n_{0}}}\right)$,
	and for $n=n_0+m$, $m\ge 1$, it is given by
	\begin{equation*}
	q(n_{0}+m,t)=\begin{cases*}
	\displaystyle\sum_{\Theta_{m}^{k}}\prod_{j=0}^{m-1}\lambda_{(n_{0}+j)_{x_{j+1}}}\mathcal{L}^{-1}\left(\frac{s^{\alpha_{n_{0}}-1}}{\prod_{l=1}^{m^{*}}(s^{\alpha_{n_{0}+(j)_{l}}}+\mu_{l})};t\right),\ \ 1\le k\le m-1,\\
	\displaystyle\sum_{\Theta_{m}^{m}}\prod_{j=0}^{m-1}\lambda_{(n_{0}+j)_{x_{j+1}}}\mathcal{L}^{-1}\left(\frac{s^{\alpha_{n_{0}}-1}}{\prod_{l=1}^{m^{*}}(s^{\alpha_{n_{0}+(j)_{l}}}+\mu_{l})};t\right),\ \  k\ge m,
	\end{cases*}
	\end{equation*}
	where 	$m^{*}=|\Lambda_{m}|$, $\mu_{l}=\sum_{i=1}^{k}\lambda_{(n_{0}+(j)_{l})_{i}}$,  $(j)_{l}$'s are given by (\ref{(j)}) and
	\begin{equation*}
\mathcal{L}^{-1}\left(\frac{s^{\alpha_{n_{0}}-1}}{\prod_{l=1}^{m^{*}}(s^{\alpha_{n_{0}+(j)_{l}}}+\mu_{l})};t\right)=	\frac{(-1)^{m^{*}-1}}{\overset{m^{*}}{\underset{l=2}{\prod}}\mu_{l}}\sum_{i=m^{*}-1}^{\infty}(-1)^{i}\sum_{\Omega^{i}_{m^{*}}}\frac{t^{\sum_{l=1}^{m^{*}}y_{l}\alpha_{n_{0}+(j)_{l}}}}{\Gamma\left(1+\sum_{l=1}^{m^{*}}y_{l}\alpha_{n_{0}+(j)_{l}}\right)}\prod_{l=1}^{m^{*}}\mu_{l}^{y_{l}}.
	\end{equation*}
	\end{theorem}
\begin{remark}
	Let $W_{1}$ denote the first waiting time of state dependent GFBP. Its distribution is given by
	\begin{equation*}
	\mathrm{Pr}\{W_{1}>t\}=\mathrm{Pr}\{\mathcal{Q}(t)=n_{0}\}=E_{\alpha_{n_{0}},1}\left(-\sum_{i=1}^{k}\lambda_{(n_{0})_{i}}t^{\alpha_{n_{0}}}\right).
	\end{equation*}
\end{remark}

 If we consider a process $\{\mathcal{Q}_{1}(t)\}_{t\ge0}$ which has the jumps of any size $j\ge 1$ then from the system (\ref{model3}) of $\{\mathcal{Q}(t)\}_{t\ge0}$, we get
 \begin{equation*}
 \frac{\mathrm{d}^{\alpha_{n}}}{\mathrm{d}t^{\alpha_{n}}}q_{1}(n,t)=-\sum_{i=1}^{\infty}\lambda_{(n)_{i}}q_{1}(n,t)+	\sum_{i=1}^{n}\lambda_{(n-i)_{i}}q_{1}(n-i,t),\ \ 0<\alpha_{n} \le 1,\ n\ge n_{0},
 \end{equation*}
  where $q_{1}(n,t)=\mathrm{Pr}\{\mathcal{Q}_{1}(t)=n\}$. Here, we assume that  $\sum_{i=1}^{\infty}\lambda_{(n)_{i}}<\infty$ for all $n\ge n_{0}$.
 
Its state probabilities are given by
	$q_{1}(n_{0},t)=E_{\alpha_{n_{0}},1}\left(-\sum_{i=1}^{\infty}\lambda_{(n_{0})_{i}}t^{\alpha_{n_{0}}}\right)$
	and 
	\begin{equation*}
	q_{1}(n_{0}+m,t)=\sum_{\Theta_{m}^{m}}\prod_{j=0}^{m-1}\lambda_{(n_{0}+j)_{x_{j+1}}}\mathcal{L}^{-1}\left(\frac{s^{\alpha_{n_{0}}-1}}{\prod_{l=1}^{m^{*}}(s^{\alpha_{n_{0}+(j)_{l}}}+\mu_{l})};t\right),\ \  m\ge 1,
	\end{equation*}
	where $\mu_{l}=\sum_{i=1}^{\infty}\lambda_{(n_{0}+(j)_{l})_{i}}$.


\begin{thebibliography}{00}	
		
\bibitem{Alipour2015}	
 Alipour, M., Beghin, L. and Rostamy, D. (2015). Generalized fractional nonlinear birth processes. {\it Methodol. Comput. Appl. Probab.} {\bf17}(3), 525-540.
 
\bibitem{Beghin2009}
Beghin, L. and Orsingher, E. (2009). Fractional Poisson processes and related planar random motions.
{\it Electron. J. Probab.} {\bf14}(61), 1790-1827.
	


\bibitem{Di Crescenzo2016}	
Di Crescenzo, A., Martinucci, B. and Meoli, A. (2016). A fractional counting process and its connection with the Poisson process. {\it ALEA Lat. Am. J. Probab. Math. Stat.} {\bf13}(1), 291-307.

\bibitem{Feller1968}
 Feller, W. (1968). An Introduction to Probability Theory and Its Applications. Vol. I, Third edition, John Wiley \& Sons, New York. 

\bibitem{Garra2015}	
Garra, R., Orsingher, E. and Polito, F. (2015). State-dependent fractional point processes. {\it J. Appl. Probab.} {\bf52}(1), 18-36.

\bibitem{Kataria2019}
Kataria, K. K. and Vellaisamy, P. (2019). On distributions of certain state-dependent fractional point processes. {\it J. Theoret. Probab.} {\bf32}(3), 1554-1580.

\bibitem{Kataria2021}	
Kataria, K. K. and Khandakar, M. (2021a).
Convoluted fractional Poisson process. {\it ALEA Lat. Am. J. Probab. Math. Stat.} {\bf18}(2), 1241-1265.



\bibitem{Kataria2021}	
Kataria, K. K. and Khandakar, M. (2021b). Generalized fractional counting process. arXiv:2106.11833.


 
\bibitem{Kilbas2006}
Kilbas, A. A., Srivastava, H. M. and Trujillo, J. J. (2006). Theory and Applications of	Fractional Differential Equations. Elsevier Science B.V., Amsterdam.

\bibitem{Orsingher2010}
Orsingher, E. and Polito, F. (2010). Fractional pure birth processes. {\it Bernoulli} {\bf16}(3), 858-881.

\bibitem{Orsingher2012}
 Orsingher, E. and Polito, F. (2012). The space-fractional Poisson process. {\it Statist. Probab. Lett.} {\bf 82}(4), 852-858.
	





\end{thebibliography}
\end{document}